\documentclass[12 pt]{article}%
\usepackage{amsmath, amsfonts, amsthm, color,latexsym}
\usepackage{amsmath, ulem}
\usepackage{amsfonts}
\usepackage{amssymb}
\usepackage{color, soul}
\usepackage{xcolor}
\usepackage[all]{xy}
\usepackage{graphicx}%
\usepackage{hyperref}
\providecommand{\U}[1]{\protect\rule{.1in}{.1in}}
\allowdisplaybreaks[4]
\newtheorem{theorem}{Theorem}[section]
\newtheorem{proposition}[theorem]{Proposition}
\newtheorem{corollary}[theorem]{Corollary}
\newtheorem{example}[theorem]{Example}
\newtheorem{examples}[theorem]{Examples}

\newtheorem{remarks}[theorem]{Remarks}
\newtheorem{lemma}[theorem]{Lemma}
\newtheorem{final remark}[theorem]{Final Remark}
\newtheorem{definition}[theorem]{Definition}
\allowdisplaybreaks[4]

\newcommand{\sol}[1]{\text{sol}(#1)}
\newcommand{\cvfa}{\overset{|\sigma|(E,E^{*})}{\longrightarrow}}

\newcommand{\cvf}{\overset{\omega}{\rightarrow}}

\newcommand {\cvfe} {\overset{\omega^\ast}{\rightarrow}}
\newcommand {\R}{\mathbb{R}}

\newcommand {\N} {\mathbb{N}}
\newcommand{\norma}[1]{\| #1 \|}
\newcommand{\conj}[2]{\left \{ {#1} \, : \, {#2} \right \}}

\begin{document}

\title{On norm-attaining positive operators between Banach lattices}
\author{Jos\'e Lucas P. Luiz \thanks{Supported by Fapemig Grant APQ-01853-23.} \, and  Vinícius C. C. Miranda \thanks{Supported by FAPESP grant and 2023/12916-1 Fapemig Grant APQ-01853-23 \newline 2020 Mathematics Subject Classification: 46B42, 46B28; 47B65
\newline Keywords: absolutely James boundaries; absolute weak topology;norm-attatining positive operators; positive weak maximizing property}}
\date{}
\maketitle

\begin{abstract}
    In this paper we study the norm-attainment of positive operators between Banach lattices. By considering an absolute version of James boundaries, we prove that: If $E$ is a reflexive Banach lattice whose order is given by a basis and $F$ is a Dedekind complete Banach lattice, then every positive operator from $E$ to $F$ is compact if and only if every positive operator from $E$ to $F$ attains its norm. An analogue result considering that $E$ is reflexive and the order in $F$ is continuous and given by a basis was proven. We applied our result to study a positive version of the weak maximizing property.
\end{abstract}

\section{Introduction}

A classical problem in Functional Analysis consists in studying the norm attainment of bounded linear operators between Banach spaces (see, e.g., \cite{acost, sheldon, dantjungrold, dantascerv, jung, lind}). One of the most known results in this direction is the James theorem which states that a Banach space $X$ is reflexive if and only if every bounded linear functional on $X$ attains its norm. Recently, by generalizing Holub's and Mujica's results (see \cite{holub, mujica}), S. Dantas, M. Jung and G. Martínez-Cervantes proved that if $X$ and $Y$ are Banach spaces such that $X$ is reflexive and the pair $(X, Y)$ satisfies the bounded compact approximation property, then every bounded linear operator from $X$ to $Y$ is compact if and only if every bounded linear operator from $X$ to $Y$ attains its norm (see \cite[Theorem B]{sheldon}).

In the Banach lattice setting, the study of norm-attaining positive linear functionals has appeared, indirectly, in the last decade. For instance, D. Ji, B. Lee, Q. Bu \cite[Lemma 3.5]{jileebu} and T. Oikhberg, M. A. Tursi \cite[Proposition 19.27]{oikhbergtursi} proved, separately, the lattice version of James theorem: An order continuous Banach latice $E$ is reflexive if and only if every positive linear functional on $E$ attains its norm. Moreover, S. Dantas, G. Martínez-Cervantes, J. D. Rodríguez-Abellán and A. Rueda-Zoca studied norm-attaining lattice-homomorphisms \cite{dantascerv}. It is also worth mentioning the recent work \cite{acosta} from M. D. Acosta and M. Soleimani-Mourchehkhorti  that studies the Bishop-Phelps-Ballobás property for positive linear functionals on Banach lattices. It is then natural to seek for a lattice version of \cite[Theorem B]{sheldon}. In our way to prove this result, we realized that some well known results on the Banach space theory have no known analogues in the Banach lattice setting, so we decided to introduce and develop some necessary results about absolute versions of James boundaries and operator topologies in Section 2. For instance, an absolute version of the Pfitzner's theorem (see \cite[Theorem 3.121]{fabian} or \cite{pfitzner}) was obtained considering absolutely weakly sequentially compact spaces.
In the end of Section 2, we proved that for every reflexive Banach lattice $E$ whose order is given by a basis and every Dedekind complete Banach lattice $F$, then every positive operator from $E$ to $F$ is compact if and only if every positive operator from $E$ to $F$ attains its norm. An analogue, considering that the order in $F$ is continuous and given by a basis was proven in the same statement.

As an interesting application of \cite[Theorem B]{sheldon} is the characterization of a Banach space with the Schur property in therms of the so-called weak maximizing property (see \cite[Theorem 3.5]{dantasjungcerv}), we explore an analogue application for our result. Following the line of investigation in the Banach lattice theory that consists in studying positive versions of well-known properties in Banach spaces, we define in Section 3 a positive version of the weak maximizing property introduced by R. Aron, D. García, D. Pellegrino and E. Teixeira in \cite{arongarciapelteix}. We investigate this new property by proving some expected results and providing examples. In the end of Section 3, we apply the result obtained in Section 2 to prove that if a Banach lattice $F$ with order continuous norm fails the positive Schur property, then there exists a reflexive Banach lattice $E$ such that the pair $(E, F)$ fails the positive weak maximizing property.

We refer the reader to \cite{alip, meyer} for background on Banach lattices and to \cite{fabian} for Banach space theory. Throughout this paper $X, Y$ denote Banach spaces and $E, F$ denote Banach lattices. 
We denote by $B_X, S_X$ and $X^\ast$ the closed unit ball, the unit sphere and the topological dual of $X$, respectively. For a subset $A \subset E$, $\text{co}(A)$, $\overline{\text{co}}(A)$ and $A^+$  denotes, respectively, the convex hull of $A$, the closed convex hull of $A$ and the positive elements of $A$. As usual, the weak topology on a Banach space $X$ shall be denoted by $\sigma(X,X^{*})$ or $\omega$, and the weak$^*$ topology on $X^*$ by $\sigma(X^{*},X)$ or $\omega^*$.

\section{Norm attaining positive operators}

Let $X$ be a Banach space and $B \subset A \subset X^\ast$. We recall that $B$ is called 
a boundary of $A$ if for each $x \in X$ there exists $x_0^\ast \in B$ such that $\displaystyle x_0^\ast(x) = \sup_{x^\ast \in A} x^\ast(x)$. In particular, $B \subset B_{X^\ast}$ is called a James boundary of $X$ if it is a boundary of the unit ball $B_{X^\ast}$, that is for each $x \in X$, exists $x_0^\ast \in B$ such that $\displaystyle x_0^\ast(x) = \sup_{x^\ast \in B_{X^\ast}} x^\ast (x)$ (see \cite[Definition 3.118]{fabian}). The set of all extreme points of $B_{X^\ast}$ is known to be an example of a James boundary set (see \cite[Fact 3.119]{fabian}). Now, we define the class of absolutely James boundaries sets in Banach lattices:

\begin{definition}
    Let $E$ be a Banach lattice and $B \subset A \subset E^\ast$. We say that $B$ is an \textbf{absolutely boundary} of $A$ if for every $x \in E$ there exists $x_0^\ast \in B$ such that $\displaystyle |x_0^\ast|(|x|) = \sup_{x^\ast \in A} |x^\ast|(|x|)$. In particular, $B \subset B_{E^\ast}$ is called an \textbf{absolutely James boundary} of $E$ if it is a absolutely boundary of the unit ball $B_{E^\ast}$, that is for each $x \in E$, exists $x_0^\ast \in B$ such that $\displaystyle |x_0^\ast|(|x|) = \sup_{x^\ast \in B_{E^\ast}} |x^\ast| (|x|)$. 
\end{definition}

The following is an easy characterization of absolutely James boundaries sets.

\begin{proposition}
    For a Banach lattice $E$ and $B \subset B_{E^\ast}$, the following are equivalent: \\
    {\rm (1)} $B$ is an absolutely James boundary of $E$. \\
    {\rm (2)} For every $x \in E$ there exists $x_0^\ast \in B$ such that $|x_0^\ast|(|x|) = \norma{x}$. \\
    {\rm (3)} For every $x \in S_E$ there exists $x_0^\ast \in B$ such that $|x_0^\ast|(|x|) = 1$. 
\end{proposition}

\begin{proof}
    (1)$\Rightarrow$(2) Let $x \in E$. Since $B$ is an an absolutely James boundary of $E$, there exists $x_0^\ast \in B$ such that $|x_0^\ast|(|x|) = \sup_{x^\ast \in B_{E^\ast}} |x^\ast|(|x|)$. On the other hand, if $J: E \to E^{\ast \ast}$ denotes the canonical embedding, then $J(|x|): E^\ast \to \R$ is a positive linear functional, and by \cite[Proposition 1.3.5]{meyer} we get that 
    $$\norma{x} = \norma{|x|} = \norma{J(|x|)} = \sup_{x^\ast \in B_{E^\ast}^+} J(|x|)(x^\ast) = \sup_{x^\ast \in B_{E^\ast}^+} x^\ast(|x|).   $$
    Since $B_{E^\ast}^+ = |B_{E^\ast}|$, we conclude that $\displaystyle \norma{x} = \displaystyle \sup_{x^\ast \in B_{E^\ast}} |x^\ast|(|x|) = |x_0^\ast|(|x|)$. 

    (2)$\Rightarrow$(3) Immediate.

    (3)$\Rightarrow$(1) If $x = 0$, then for every $x_0^\ast \in B$, we get that $\displaystyle |x_0^\ast|(|x|) = 0 = \sup_{x^\ast \in B_{E^\ast}} |x^\ast|(|x|)$. Now, letting $x \in E$ with $x \neq  0$, there exists by the assumption $x_0^\ast \in B$ such that $ \displaystyle |x_0^\ast|\left (\frac{|x|}{\norma{x}} \right ) = 1$, which implies that $|x^\ast|(|x|) = \norma{x}$. Considering again the canonical embedding $J(|x|)$, we get from that $\displaystyle \norma{x} = \norma{J(|x|)} = \sup_{x^\ast \in B_{E^\ast}} |x^\ast|(|x|)$, and we are done.
\end{proof}

Our first example of an absolutely James boundary set is the set of order extreme points of $B_{E^\ast}$. First, we recall that a point $a$ of $A \subset E$ is said to be an order extreme point of $A$ if for all $x_0, x_1 \in A$ and $t \in (0,1)$ the inequality $a \leq t x_0 + (1-t) x_1$ implies that $x_0 = a = x_1$ (see \cite{oikhbergtursi}).

  \begin{example} \label{ex1}
        The set $\text{\rm OEP} (B_{E^\ast})$ of all order extreme points of $B_{E^\ast}$ is an absolutely James boundary of $E$. {\normalfont Indeed, for a given $x \in E$, we define $H = \displaystyle \conj{f \in E^\ast}{f(|x|) = \norma{x}}$.        
        We notice that $H\neq \emptyset$, by Hahn-Banach Theorem, and:
        
        $\bullet$ $H$ is weak* closed: If $(f_\alpha)_\alpha\subset H$ is a net such that $f_\alpha \cvfe f$ in $E^\ast$, then 
        $$ f(|x|) = \lim_\alpha f_\alpha(|x|) = \lim_{\alpha} \norma{x} = \norma{x}, $$
        which implies that $f \in H$.

        $\bullet$ $H$ is a supporting manifold of $B_{E^\ast}$ in the sense of  \cite[Definition 3.60]{fabian} by \cite[Lemma 3.61]{fabian}, because the linear functional $J_{E}(|x|): E^\ast \to \R$ attains its norm on $B_{E^\ast}$.

        Now, it follows from \cite[Proposition 3.64]{fabian} that $H$ contains an extreme point of $B_{E^\ast}$. If $f$ is such a point, then $f(|x|) = \norma{x}$, and consequently
        $$ \norma{x} = f(|x|) \leq |f(|x|)| \leq |f|(|x|) \leq \norma{x}. $$
        However, $|f|$ is an order extreme point of $B_{E^\ast}$ (see \cite[Theorem 19.2]{oikhbergtursi}).
        }
    \end{example}

Our next example of an absolutely James boundary set will be necessary in the proof of Lemma \ref{lemaaux}. Before we proceed, it is important to observe that if a positive operator $T: E \to F$ is norm-attaining, then $T$ attains its norm on a positive normalized vector. Indeed, if $T$ is norm-attaining, there exists $x \in S_E$ such that $\norma{T(x)} = \norma{T}$. The positivity of $T$ yields that
$$ \norma{T} = \norma{T(x)} \leq \norma{T(|x|)} \leq \norma{T} \cdot \norma{|x|} = \norma{T}, $$
which implies that $\norma{T(|x|)} = \norma{T}$.

\begin{example} \label{ex2}
    Let $E$ and $F$ be two Banach lattices with $F$ being Dedekind complete. If every positive operator $T: E \to F$ attains its norm, then the set
    $$B = \conj{x \otimes y^\ast}{x \in S_{E}^+, \, y^\ast \in S_{F^\ast}^+}$$
    is an absolutely James boundary of $\mathcal{L}^r(E; F)$. {\normalfont Indeed, for every $T \in S_{\mathcal{L}^r(E; F)}$, we have by the assumption that $|T|$ attains its norm, and hence there exists $x \in S_{E}^+$ such that $\norma{|T|(x)} = \norma{|T|} = \norma{T}_r = 1$. Since $|T|(x) \geq 0$, there exists $y^\ast \in S_{F^\ast}^+$ such that $y^\ast(|T|(x)) = 1$ (Apply \cite[Corollary 1.3]{lotz} to the sublattice $[|T|(x)]$ and the positive linear functional $g(\lambda \, |T|(x)) = 1$). Finally, taking $f = x \otimes y^\ast$, we get that $0 \leq f \in \left ( \mathcal{L}^r(E; F) \right )^\ast$, and hence 
    $|f|(|T|) = f(|T|) = y^\ast(|T|)(x) = 1$.}
\end{example}

In \cite{pfitzner}, H. Pfitzner showed that if $B$ is a James boundary of a Banach space $X$ and $K \subset X$ is a norm bounded $\sigma(X, B)$-compact set, then $K$ is weakly compact. In Theorem \ref{teosimons} we are going to prove an absolute version of this result considering sequentially compact sets in the absolute weak topology. Let us recall that for a nonempty subset $A$ of $E^\ast$, the absolute weak topology $|\sigma|(E,A)$ is the locally convex-solid topology on $E$ generated by the family $\{p_{x^{*}}: x^{*} \in A\}$ of lattice seminorms, where $p_{x^{*}}(x) = |x^{*}| (|x|)$ for all $x \in E$ and $x^{*} \in A$. The compact and the sequentially compact sets for this topology were studied by the authors and G. Botelho in \cite{botlumir}. It will be proved below that, for an absolutely James boundary $B$ of $E$, if $K \subset E$ is norm bounded and sequentially $|\sigma|(E, B)$-compact, then $K$ is absolutely weakly sequentially compact, that is $K$ is sequentially $|\sigma|(E, E^\ast)$-compact. First, we recall the Simons' Inequality \cite{oja} and \cite[Lemma 3.123]{fabian}:

\begin{lemma}[Simons' Inequality]\label{Simonsinequality}
    Let $A$ be an nonempty set and let $(x_n)_n$ be a uniformly bounded sequence of functions defined on $A$. Let $D$ be a subset of $A$ such that for every sequence $(\lambda_n)_n$, $\lambda_n\geq 0$, satisfying $\displaystyle\sum_{n=1}^\infty \lambda _n=1$, the vector $\displaystyle\sum_{n=1}^\infty \lambda _n x_n$ attains its supremum over $D$, that is, there exists $d\in D$ such that 
    \begin{equation}\label{condition1}
    \sup\left\{\displaystyle\sum_{n=1}^\infty \lambda_n x_n(y):~y\in A\right\}=\displaystyle\sum_{n=1}^\infty \lambda _n x_n(d).
    \end{equation}
    Then $$\displaystyle\sup_{d\in D}\{\displaystyle\limsup_n{x_n(d)}\}\geq \displaystyle\inf_{x \in  \text{\rm co}(\{x_n;~n\in\mathbb{N}\})}\{\displaystyle\sup_{y\in A} x(y)\}$$
\end{lemma}

In the following, we will apply Simons' inequality for $A = |B_{E^\ast}|$, $D = |B|$, where $B$ is an absolutely James boundary of $E$ and the absolute value of a bounded sequence in $E$ acting as functionals on $A$.

\begin{lemma} \label{lemasimons}
    Let $B$ be an absolutely James boundary of a Banach lattice $E$. If $(x_n)_n$ is a bounded sequence in $E$ such that $|g|(|x_n|)\longrightarrow 0$ for all $g\in B$, then $|f|(|x_n|)\longrightarrow 0$ for all $f\in E^*$. 
\end{lemma}
\begin{proof}
    For the sake of contradiction, we  assume that there exists $f \in B_{E^\ast}$ such that $|f|(|x_n|)\not\rightarrow 0$. Thus, there exists $\varepsilon>0$ and a subsequence $(z_k)_k$ of $(x_n)_n$ such that $|f|(|z_k|)>\varepsilon$ for every $k\in \mathbb{N}$. Let us prove that $D = |B|$ satisfies the condition (\ref{condition1}) on Simons' inequality for the sequence $(|z_k|)_k$ and $A = |B_E^\ast|$. Indeed, for every real sequence $(\lambda_k)_k, \, \lambda_k \geq 0$, satisfying $\sum_{\lambda = 1}^\infty \lambda_k = 1$, the vector $\sum_{k=1}^\infty \lambda_k |z_k|$ can be considered a positive linear functional on $E^\ast$, and by \cite[Proposition 1.3.5]{meyer}
    $$ \norma{\sum_{k=1}^\infty \lambda_k |z_k|} = \sup_{f \in B_{E^\ast}^+}  \sum_{k=1}^\infty \lambda_k|z_k| (f)  =  \sup_{f \in B_{E^\ast}}  \sum_{k=1}^\infty \lambda_k|z_k| (|f|).$$
    On the other hand, since $B$ is an absolutely James boundary of $E$, there exists $g \in B$ such that 
    $$ \sup_{f \in B_{E^\ast}}  \sum_{k=1}^\infty \lambda_k|z_k| (|f|) = \norma{\sum_{k=1}^\infty \lambda_k |z_k|} = |g| \left ( \sum_{k=1}^\infty \lambda_n |z_k| \right ) = \sum_{k=1}^\infty \lambda_k |z_k|(|g|). $$
    Now, it follows from the assumption and the Simons' inequality that
    $$ \inf_{x \in  \text{\rm co}(\{|z_k|\})}\{\displaystyle \sup_{f\in B_{E^\ast}} |f|(x)\} \leq  \sup_{g \in B} \, 
        \{\displaystyle\limsup_k{|g|(|z_k|)}\} = 
        \sup_{g \in B} \, 
        \{\displaystyle\lim_k{|g|(|z_k|)}\}
        = 0, $$
   and hence there exists $ x \in  \text{\rm co}(\{|z_k|\})$ such that $\displaystyle\sup_{f\in B_{E^*}} |f|(x)<\varepsilon$, which is a contradiction with
   $ \displaystyle |f|(|z_k|) > \varepsilon   $ for every $k \in \N$.
\end{proof}

Now, we prove the absolute version of Pfitzner's theorem.

\begin{theorem} \label{teosimons}
    Let $E$ be a Banach lattice and let $B$ be an absolutely boundary of $B_{E^\ast}$. If $K \subset E$ is norm-bounded and sequentially $|\sigma|(E, B)$-compact, then $K$ is absolutely weakly sequentially compact. Moreover, if $E$ is separable or $B_{E^{**}}$ is absolutely weak$^*$ compact, then $K$ is absolutely weakly compact.
\end{theorem}

\begin{proof}
If $(x_n)_n \subset K$, we have by the assumption that there exist a subsequence $(x_{n_k})_k$ of $(x_n)_n$ and $x \in E$ such that $|g|(|x_{n_k} - x|) \to 0$ for every $g \in B$. By Lemma \ref{lemasimons}, we get that $|f|(|x_{n_k} - x|) \to 0$ for every $f \in E^\ast$, proving that $K$ is is absolutely weakly sequentially compact. Moreover,  if $E$ is separable or $B_{E^{**}}$ is absolutely weak$^*$ compact, then $K$ is absolutely weakly compact by \cite[Theorem 2.9]{botlumir}.
\end{proof}

As announced in the Introduction, absolute versions of the strong operator topology and the weak operator topology will be necessary to proof our results. First, we recall that the strong operator topology ($SOT$, for short) and the weak operator topology ($WOT$, for short) in $\mathcal{L}(X; Y)$ are defined by the following basic neighborhoods, respectively:
$$ N(T; A,\varepsilon)=\conj{S \in \mathcal{L}(X; Y)}{\norma{T(x) - S(x)} < \varepsilon \, \text{for every } x \in A  } $$
and
$$N(T; A, B,\varepsilon)= \conj{S \in \mathcal{L}(X; Y)}{y^\ast(T(x) - S(x)) < \varepsilon \, \, \text{for every } y^\ast \in B, x \in A } $$
where $\varepsilon > 0$, $A \subset X, \, B \subset X^\ast$ are arbitrary finite sets and $T\in \mathcal{L}(X;Y)$.

Let $E$ and $F$ be Banach lattices with $F$ being Dedekind complete.
By the one hand, we define the \textbf{absolute strong operator topology} ($|SOT|$, for short) in $\mathcal{L}^r(E; F)$ by the following basic neighborhoods
$$ N(T; A, \varepsilon) = \conj{S \in \mathcal{L}^r(E; F)}{\norma{|T - S|(x)} < \varepsilon \, \text{for every } x \in A}, $$
where $A \subset E^+$ is an arbitrary finite set and $\varepsilon > 0$. 
Thus, a net $(T_\alpha)_\alpha \subset \mathcal{L}^r(E; F)$ converges to $T$ in the $|SOT|$ if and only if $(|T_\alpha - T|(x))_\alpha$ converges to $0$ for every $x \in E^+$. 
On the other hand, the \textbf{absolute weak operator topology} in $\mathcal{L}^r(E; F)$ is defined by the following basic neighborhoods
$$ N(T; A, B, \varepsilon) = \conj{S \in \mathcal{L}^r(E; F)}{y^\ast(|T - S|(x)) < \varepsilon \, \text{for every } y^\ast \in B, x \in A}, $$
where $A \subset E^+$ and $B \subset (F^\ast)^+$ are arbitrary finite sets and $\varepsilon > 0$. In particular, a net $(T_\alpha)_\alpha \subset \mathcal{L}^r(E; F)$ converges to $T$ in the absolute weak operator topology ($|WOT|$, in short) if and only if $(y^\ast(|T_\alpha - T|(x)))_\alpha$ converges to $0$ for every $x \in E^+$ and $y^\ast \in (F^\ast)^+$. 

\begin{remarks} {\rm
     (1) It is important to notice that the $|SOT|$ coincides with the absolute weak* topology on $E^\ast$ whenever $F = \R$, while the $|WOT|$ coincides with the absolute weak topology in $F$ whenever $E = \R$. Thus, the $|SOT|$ and $|WOT|$ are generalizations of the well known absolute weak* topology and the absolute weak topology, respectively. 

     \smallskip

    \noindent (2) The topology $|SOT|$ is stronger than topology $SOT$. Indeed, if $(T_\alpha)_\alpha \subset \mathcal{L}^r(E; F)$ is a net converging to $T$ in the $|SOT|$, we have that $(T_\alpha)_\alpha$ converges to $T$ in the $SOT$, because $|T_\alpha(x) - T(x)| \leq |T_\alpha - T|(x)$ holds for every $x \in E^+$ and every $\alpha$. To see that this two topologies are different in general, recall that for $F = \R$, the $|SOT|$ coincides with the absolute weak* topology and the $SOT$ coincides with the weak* topology in $E^\ast$.

    \smallskip

    \noindent (3) By adapting the argument from the item above, we have that the $|WOT|$ is stronger than $WOT$. These topologies are not the same in general, because the $|WOT|$ coincides with the absolute weak topology and the $SOT$ coincides with the weak topology in $F  = \mathcal{L}^r(\R; F)$.
    
    \smallskip 
    \noindent (4) It is easy to see that the $|SOT|$ is stronger than $|WOT|$. Moreover, these topologies do not coincide because the absolute weak* topology and the absolute weak topology are not the same in general.

    }
\end{remarks}

The following lemma will be necessary in the proof of Theorem \ref{teoprinc}.


\begin{lemma} \label{lemaaux}
Let $E$ and $F$ be Banach lattices with $F$ being Dedekind complete. If there exists a norm-closed convex set $C \subset \mathcal{L}^r(E; F)$ which is not sequentially $|SOT|$-closed, then there exists a positive non-norm attaining operator $T: E \to F$. 
\end{lemma}

\begin{proof} 
   Since $C$ is not sequentially $|SOT|$-closed, there exists a sequence of operators $(T_n)_n \subset C$ such that $(T_n)_n$ converges to an operator $T \notin C$ in the $|SOT|$, and hence in the $|WOT|$. Setting $K = \conj{T_n - T}{n \in \N}$, we get that 
    
    \smallskip
    \noindent (i) $K \cup \{0\}$ is  sequentially $|WOT|$-compact. \\
    (ii) 
    If $0 \in \overline{\text{co}}(K)$, there exists $(S_n)_n \subset \text{co}(K)$ such that $\norma{S_n}_r \to 0$. For each $n \in \N$, there exists $a_1, \dots, a_{k_n} \geq 0$ such that $\sum_{i=1}^{k_n} a_i = 1$ and $S_n = \sum_{i=1}^{k_n} a_i (T_i - T) = \sum_{i=1}^{k_n} a_i T_i - a_iT$. Since $C$ is convex, $\sum_{i=1}^{k_n} a_i T_i \in C$ for every $n \in \N$. So, letting $R_n = \sum_{i=1}^{k_n} a_i T_i$, we have that 
    $$ \lim_{n \to \infty} \norma{R_n - T}_r = \lim_{n \to \infty} \norma{S_n}_r = 0. $$
    As $C$ is assumed to be norm-closed, we get that $T \in C$, a contradiction. Thus, $0 \notin \overline{\text{co}}(K)$
    

    \noindent (iii) $K$ is a norm-bounded subset of $\mathcal{L}^r(E; F)$. Indeed,
    since $(T_n)_n$ converges to $T$ in the $|SOT|$, $(|T_n - T|(x))_n$ is a bounded sequence in $F$ for every $x \in E$, and so by the  Banach-Steinhauss theorem, there exists $C > 0$ such that 
    $$ \norma{T_n - T}_r = \norma{|T_n - T|} \leq C  \quad \text{for every $n \in \N$}.$$

    \smallskip

    Suppose that every positive operator $T\colon E\longrightarrow F$ attains its norm. Defining $B = \conj{x \otimes y^\ast}{x \in S_{E}^+, \, y^\ast \in S_{F^\ast}^+}$, we have by Example \ref{ex2} that $B$ is an absolutely James boundary of $\mathcal{L}^r(E; F)$, and hence $K$ is sequentially $|\sigma|(E, B)$-compact. Indeed, since $K \cup \{0\}$ is sequentially $|WOT|$-compact, for every sequence $(S_n)_n \subset K$, there exists a subsequence $(S_{n_k})_k$ converging to $S \in K \cup \{0\}$, that is $y^\ast(|S_{n_k} - S|)(x) \to 0$ for every $x \in E^+$ and $y \in (F^\ast)^+$. In particular, $|f|(|S_{n_k} - S|) \to 0$ for every $f \in B$, proving that $K \cup \{0\}$ is sequentially $|\sigma|(E, B)$-compact.
    Now, by Theorem \ref{teosimons}, we obtain that $K \cup \{0\}$ is an absolutely weakly sequentially compact set, which implies that $K \cup \{0\}$ is sequentially $|\sigma|(E, E^\ast)$-closed.
    Since the sequential  closure is contained in the closure, we get that $0 \in \overline{K}^{|\sigma|(E, E^\ast)}$, which yields that $0 \in \overline{co}^{|\sigma|(E, E^\ast)}(K) = \overline{co}(K)$ - recall from \cite[p. 10]{botlumir} that the norm closure and the weak absolute closure coincide for convex sets, - a contradiction.
\end{proof}


The following is an adaptation of the argument presented in \cite[p. 2464]{bupams}, and present a  approximation result of the identity by positive compact operators.

\begin{lemma} \label{lemabu}
 Let $G$ be a Banach lattice with order continuous norm whose order is defined by a basis. For every compact subset $K \subset G$ and $\varepsilon > 0$ there exists a positive compact operator $S: G \to G$ such that $S \leq id_G$, $\norma{|S - id_G|(x)} < \varepsilon$ for every $x \in K$ and $\norma{S}_r \leq 1$.
\end{lemma}

\begin{proof}
     Let $K \subset G$ be a compact subset and $\varepsilon > 0$. By the assumption, we may see $G$ as a sublattice of $\R^{\N}$. Thus,
    proceeding as in \cite[p. 2464]{bupams}, there exists a finite rank (hence compact) operator $S: G \to G$ such that $S \leq id_G$, $\norma{|S - id_G|(x)} < \varepsilon$ for every $x \in K$ and $\norma{S}_r \leq 1$. Moreover, it is not difficult to see that this operator is positive considering its construction.  
\end{proof}

The following lemma is necessary in order to prove our main result.

\begin{lemma} \label{lemaseq}
    If $G$ is a Banach lattice with order continuous norm whose order is defined by a basis, then there exists a sequence $(S_n)_n \subset B_{\mathcal{K}^+(G; G)}$ such that $\displaystyle \lim_{n \to \infty} \norma{|S_n - id_{G}|(x)} = 0$ for every $x \in G$. 
\end{lemma}

\begin{proof}
    By the assumption, there exists a basis $(x_n)_n$ of $G$ whose order structure is given coordinate-wise. Fix $n \in \N$. Applying Lemma \ref{lemabu} to the compact $C_n = \{x_1, \dots, x_n\}$ and $\varepsilon = 1/2^n$, there exists $S_n \in B_{\mathcal{K}^+(G; G)}$ such that 
 $$\norma{|S_n - id_G|(x_i)} < 1/2^{n} \quad \text{for every $i=1, \dots, n$}.$$ 
 Now, fix $x = \displaystyle\sum_{i=1}^\infty a_i x_i$ and let $\delta > 0$. By the one hand, there exists $n_0 \in \N$ such that $\norma{\displaystyle\sum_{i=n_0 + 1}^\infty a_i x_i} < \displaystyle \frac{\delta}{4}$. On the other hand, there exists $N > n_0$ such that
 $\displaystyle \frac{1}{2^N} < \frac{\delta}{2 \sum_{i=1}^{n_0} |a_i|}$. Thus
\begin{align*}
    \norma{|S_N - id_G|(x)} & \leq \sum_{i=1}^{n_0} |a_i| \norma{|S_N - id_G|(x_i)} + \norma{|S_N - id_G|} \norma{\sum_{i=n_0 + 1}^\infty a_i x_i} \\
    & < \sum_{i=1}^{n_0} |a_i| \frac{1}{2^N} + 2 \frac{\delta}{4} < \delta,
\end{align*}
proving that $\displaystyle \lim_{n \to \infty} \norma{|S_n - id_G|(x)} = 0$ for every $x \in G$. 
\end{proof}

\begin{theorem} \label{teoprinc}
     Let $E$ be a reflexive Banach lattice and let $F$ be a Dedekind complete Banach lattice. Consider the following conditions: \\
{\rm (1)} Every positive operator $T: E \to F$ is compact. \\
{\rm (2)} Every positive operator $T: E \to F$ attains its norm. \\
{\rm (3)} $B_{\mathcal{K}^+(E; F)}$ is sequentially closed in the absolutely strong operator topology. \\
Then {\rm (1)}$\Rightarrow${\rm(2)}$\Rightarrow${\rm(3)}. In addition, if the order of $E$ is given by a basis or $F$ has order continuous norm whose order is defined by a basis, then {\rm(3)}$\Rightarrow${\rm(1)}.
\end{theorem}

\begin{proof}
    (1)$\Rightarrow$(2) Let $T: E \to F$ be a positive operator and let $(x_n)_n \subset B_E$ such that $\displaystyle \lim_{n \to \infty} \norma{T(x_n)} = \norma{T}$. 
    Since $T$ is compact and $E$ is reflexive, there exists a subsequence $(x_{n_k})_k$ such that $(T(x_{n_k}))_k$ converges to $T(x)$ in $F$ for some $x \in E$. Thus
    $$ \norma{T(x)} = \norma{T(\lim_{k \to \infty} x_{n_k})} = \lim_{k \to \infty} \norma{T(x_{n_k})} = \norma{T}, $$
    proving that $T$ attains its norm.

    (2)$\Rightarrow$(3) This implication follows from Lemma \ref{lemaaux}.

    We now prove that (3)$\Rightarrow$(1). First, we assume that the order of $E$ is given by a basis. Thus, it follows from Lemma \ref{lemaseq} that there exists a sequence $(S_n)_n \subset B_{\mathcal{K}^+(E; E)}$ such that $\displaystyle \lim_{n \to \infty} \norma{|S_n - id_{E}|(x)} = 0$ for every $x \in E$. Letting $T: E \to F$ be a positive operator and defining $T_n = T \circ S_n, n \in \N$, we get that $(T_n)_n$ is a sequence of positive compact operators with $\norma{T_n} \leq \norma{T}$ for every $n \in \N$. Moreover, since
    $$ |(T_n - T)(x)| = |(T \circ S_n - T \circ id_E)(x)| = |T((S_n - id_E)(x))| \leq T(|S_n - id_E|(x)) $$
    for every $n \in \N$ and $x \in E^+$, we get that
    $|T_n- T|(x) \leq T(|S_n - id_E|(x))$ for every $x \in E^+$, which implies that
    $$ \norma{|T_n - T|(x)} \leq \norma{T(|S_n - id_E|(x))} \leq \norma{T} \norma{|S_n - id_E|(x)} \to 0 $$
 when $n \to \infty$ for all $x \in E^+$, proving that $(T_n)_n$ converges to $T$ in the $|SOT|$. Therefore, as we assumed that $B_{\mathcal{K}^+(E; F)}$ is sequentially closed in the absolutely strong operator topology, we conclude that $T$ is compact.

 Assuming that $F$ has order continuous norm and its order is defined by a basis, there exists, by Lemma \ref{lemaseq}, a sequence $(S_n)_n \subset B_{\mathcal{K}^+(F; F)}$ such that $\displaystyle \lim_{n \to \infty} \norma{|S_n - id_{F}|(y)} = 0$ for every $y \in F$. For a positive operator $T: E \to F$, define $T_n = S_n \circ T$ and proceed as in the last case to obtain that $(T_n)_n$ converges to $T$ in the $|SOT|$. Finally, the assumption yields $T$ is compact.
\end{proof}

The next remark highlights a striking difference between classical norm-attaining operators and norm-attaining positive operators.

\begin{remarks} \label{obs}
{\rm (1)} It is important to observe that Theorem \ref{teoprinc} can be applied to $E = \ell_p$ and $F = L_q(\mu)$ whenever $1 \leq q < p\leq 2$, because every positive linear operator $T: E \to F$ is compact (\cite[Theorem 4.9]{chenwick}). Note, however, that there exist non-compact linear operators from $\ell_p$ into $L_q(\mu)$ by \cite[Theorem 4.7]{chenwick}. 
Thus, it follows from \cite[Theorem B]{sheldon} that there exists a non-norm attaining operator from $\ell_p$ into $L_q(\mu)$, while every positive operator from $\ell_p$ into $L_q[0,1]$  attains its norm.\\
{\rm (2)} Let $E = L_p(\mu)$ and $F = \ell_q$ with $2 \leq q < p < \infty$. By \cite[Theorem 4.7]{chenwick}, there exists a non-compact bounded linear operator from $L_p(\mu)$ into $\ell_q$, and so by \cite[Theorem B]{sheldon} there exists a non-norm attaining operator from $L_p(\mu)$ into $\ell_q$. Nevertheless, every positive operator from $L_p(\mu)$ into $\ell_q$ is norm attaining from Theorem \ref{teoprinc} and \cite[Theorem 4.9]{chenwick}. \\
{\rm (3)} If $E$ is an $AM$-space with an order unit $e$, i.e. $B_E = [-e,e]$, then every positive linear operator from $E$ into any Banach lattice $F$ attains its norm in $e$ (see \cite[Exercise 2, p. 270]{alip}). However, since $E$ is not reflexive, there exists a (non-positive) continuous linear functional $f \in E^*$ that does not attain its norm. Thus, for every Banach lattice $F$ and every $0 \neq y \in F$, the non-positive linear operator $T(x) = f(x) y$ does not attain its norm.
\end{remarks}

Our next result is an application of Theorem \ref{teoprinc} and \cite[Theorem 4.3]{botmirru} for $n=1$.

\begin{corollary}
    Let $E$ be a reflexive Banach lattice and let $F$ be an infinite dimensional Banach lattice with order continuous norm whose order is defined by a basis. The following are equivalent: \\
    {\rm (1)} $\mathcal{L}^r(E; F)$ contains no copy of $c_0$. \\
    {\rm (2)} Every positive linear operator from $E$ to $F$ is compact. \\
    {\rm (3)} Every positive linear operator from $E$ to $F$ attains its norm. 
\end{corollary}

\begin{proof}
    Notice that we may apply \ref{teoprinc} and \cite[Theorem 4.3]{botmirru}, because reflexive Banach spaces fail the dual positive Schur property and Banach lattices whose order are defined by a basis are atomic.
\end{proof}

\section{The positive weak maximizing property}

The purpose of this section is to introduce and investigate a positive version of the so-called weak maximizing property. Recall that a pair of Banach spaces $(X, Y)$ is said to have the weak maximizing property if a bounded linear operator $T: X \to Y$ is norm-attaining, whenever there exists a non-weakly null maximizing sequence for $T$. By a maximizing sequence for a bounded linear operator $T$ we mean a sequence $(x_n)_n \subset S_X$ such that $\displaystyle \lim_{n \to \infty} \norma{T(x_n)} = \norma{T}$. This property was introduced in \cite{arongarciapelteix}, and then studied in \cite{dantasjungcerv, garcia-lirola, jung}. In the environment of Banach lattices, the following definition arises naturally by considering positive sequences in the definition of the $WMP$.

\begin{definition} \label{wmp1}
    We say that a pair of Banach lattices $(E, F)$ has the \textbf{positive weak maximizing property} ($WMP^+$) if a positive operator $T: E \to F$ is  norm-attaining whenever there exists a positive non-weakly null maximizing sequence of $T$.
\end{definition}

\begin{examples}\label{ex3}{\rm (1)} It is easy to check that if $(E, F)$ has the $WMP$, then it has the $WMP^+$.\\ 
{\rm (2)} If every positive operator from $E$ into $F$ is norm-attaining, then $(E, F)$ has the $WMP^+$. So, as noticed in Remarks \ref{obs}: 

{\rm (i)} $(\ell_p, L_q(\mu))$ has the $WMP^+$ whenever $1 \leq q < p \leq 2$. 

{\rm (ii)} $(L_p(\mu), \ell_q)$ has the $WMP^+$ whenever $2 \leq q < p < \infty$. 

{\rm (iii)} $(E, F)$ has the $WMP^+$ whenever $E$ is an $AM$-space with an order unit. \\
{\rm (3)} As noticed in the proof of \cite[Theorem 3.2]{dantasjungcerv}, whenever $p > 2$, the pair $(L_p[0,1], \ell_2)$ fails the $WMP$, while it has the $WMP^+$ by item {\rm (2)(ii)}.  Moreover, if $E$ is an $AM$-space with an order unit, then $E$ is not reflexive, and by \cite[Corollary 2.5]{arongarciapelteix}, $(E, Y)$ cannot have the $WMP$ for any Banach space $Y$, which is a contrast with item {\rm (2)(iii)}.
After the proof of Theorem \ref{wmp5}, we will present another example of pair $(E, F)$ with the $WMP^+$, failing the $WMP$. 
\end{examples}

We begin our discussion by proving some elementary facts about the behavior of the $WMP^+$ considering closed sublattices and projection bands.

\begin{proposition} \label{wmp4}
    Given two Banach lattices $E$ and $F$, the following conditions hold: \\
    {\rm (a)} If $(E, F)$ has the $WMP^+$, then so does $(E, F_1)$ for every closed sublattice $F_1$ of $F$. \\
 {\rm (b)} If $(E, F)$ has the $WMP^+$, then so does $(E_1, F)$ for every
sublattice $E_1$ of $E$ that is the range of a positive projection. \\
   {\rm (c)} If there exists a non-norm attaining positive operator $T: E \to F$ with $\norma{T} = 1$, then  $(E \oplus_p \R, F \oplus_q \R)$ fails the $WMP^+$ for all $1 \leq q < p < \infty$.
\end{proposition}

\begin{proof}
    (a) Let $T: E \to F_1$ be a positive operator with a positive non-weakly null maximizing sequence $(x_n)_n \subset E$. Considering the canonical inclusion $j: F_1 \to F$, we get that $j$ is a positive operator, and hence $S = j \circ T: E \to F$ is also a positive operator with a positive non-weakly null maximizing sequence $(x_n)_n$:
    $$ \norma{S} = \norma{T} = \lim \norma{T(x_n)} = \lim \norma{j(T(x_n))} = \lim \norma{S(x_n)}. $$
    Since $(E, F)$ has the $WMP^+$, we have that $S$ attains its norm on a positive vector $x \in S_{E^+}$, which implies that
    $ \norma{T} = \norma{S} = \norma{S(x)} = \norma{j(T(x))} = \norma{T(x)}, $
    proving that $(E, F_1)$ has the $WMP^+$.

    (b) Let $T: E_1 \to F$ be a positive operator with a positive non-weakly null maximizing sequence $(x_n)_n \subset E_1$.
    By the assumption, there exists a positive projection $P: E \to E$ such that $P(E) = E_1$. Thus, defining $S = T \circ P: E \to F$, we get that $S$ is a positive operator such that $\norma{S} = \norma{T}$. On the other hand, since $(x_n)_n$ is also a non-weakly null sequence in $E$, we have that $(x_n)_n$ is a non-weakly null maximizing sequence of $S$. Thus, the assumption on $(E, F)$ having the $WMP^+$ yields that $S$ attains its norm on a positive vector $x \in S_{E^+}$, which implies that $T$ attains its norm on the positive vector $P(x)$.

    (c) Define $S: E \oplus_p \R \to F \oplus_q \R$ by $S(x, a) = (Tx, a)$ for every $x \in E$ and $a \in \R$. As $T \geq 0$, $S \geq 0$. 
    Moreover, there exists a positive maximizing sequence $(x_n)_n$ of $T$. Indeed, if $(z_n)_n$ is a maximizing sequence of $T$, then $(|z_n|)_n$ is also a maximizing sequence of $T$ because the inequality $\norma{Tz_n} \leq \norma{T |z_n|} \leq \norma{T} = 1$ holds for every $n \in \N$. Thus, the proof of \cite[Main theorem]{dantasjungcerv} yields that $S$ is a non-norm attaining operator with a positive non-weakly null maximizing sequence. Therefore, $(E \oplus_p \R, F \oplus_q \R)$ fails the $WMP^+$.
\end{proof}


\begin{proposition} \label{wmp3}
    A Banach lattice with order continuous norm $E$ is reflexive if and only if the pair $(E, F)$ has the $WMP^+$ for some Banach lattice $F \neq \{0\}$.
\end{proposition}

\begin{proof}
    Assume that the pair $(E, F)$ has the $WMP^+$ for some Banach lattice $F \neq \{0\}$.
    If $E$ is a non-reflexive Banach lattice, there exists a positive linear functional $x^\ast \in E^\ast$ that does not attain its norm (see \cite[Proposition 19.27]{oikhbergtursi}). Take $y \in S_{F^+}$ and define $T: E \to F$ by $T(x) = x^\ast(x) y$. In particular, $T$ is a positive  which is not norm attaining. Thus, the assumption implies that every maximizing sequence $(x_n)_n$ for $T$ is absolutely weakly null, and hence $T(x_n) \to 0$ in $F$ because $T$ is compact. For such a sequence $(x_n)_n$ we get that
    \begin{align*}
        \norma{T} & = \lim_{n \to \infty} \norma{T(x_n)} = 0,
    \end{align*}
    which is a contradiction.
    The converse is immediate, because every linear functional on a reflexive space attains its norm. 
\end{proof}

If $E$ is an infinite dimensional  $AM$-space with an order unit, then $E$ cannot have order continuous norm and $(E, F)$ has the $WMP^+$ for every Banach lattice $F$ by Examples \ref{ex3}(2)(iii). This shows that Proposition \ref{wmp3} fails without assuming that $E$ has order continuous norm.

Now, we prove that the positive non-weakly null maximizing sequence in the definition of the $WMP^+$ may be replaced by a non-absolutely weakly null maximizing sequence.

\begin{proposition} \label{wmp2}
    A pair of Banach lattices  $(E, F)$ has the $WMP^+$ if and only if a positive operator $T: E \to F$ is  norm-attaining whenever there exists a non-absolutely weakly null maximizing sequence of $T$.
\end{proposition}

\begin{proof}
    Assume that $(E, F)$ has the $WMP^+$.  Let $T: E \to F$ be a positive operator and $(x_n)_n \subset E$ be a non-absolutely weakly null maximizing sequence of $T$, that is $(x_n)_n \subset S_E$ is a non-absolutely weakly null sequence such that $\displaystyle \lim_{n \to \infty} \norma{T(x_n)} = \norma{T}$. In particular, $(|x_n|)_n$ is a positive non-weakly null normalized sequence such that $|T(x_n)| \leq T(|x_n|)$ for every $n \in \N$, which implies that
    $ \norma{T(x_n)} \leq \norma{T(|x_n|)} \leq \norma{T} $
    for every $n \in \N$. Taking $n \to \infty$ in the above inequality, we get that $\displaystyle \lim_{n \to \infty} \norma{T(|x_n|)} = \norma{T}$. By the assumption, we obtain that $T$ is norm-attaining.

    Suppose now that every positive operator $T: E \to F$ is norm-attaining whenever there exists a non-absolutely weakly null maximizing sequence of $T$. To prove that $(E, F)$ has the $WMP^+$, let $T: E \to F$ be a positive operator and let $(x_n)_n$ be a positive non-weakly null maximizing sequence of $T$. Since $(x_n)_n$ is a positive sequence, we get that it cannot be absolutely weakly null as well, proving the thesis. 
\end{proof}

It was proven by S. Dantas, M. Jung and G. Martínez-Cervantes that a Banach space $Y$ has the Schur property if and only if the pair $(X, Y)$ has the $WMP$ for every reflexive space $X$ (see \cite[Theorem 3.5]{dantasjungcerv}). It is a natural question to seek if there exists a similar characterization for the $WMP^+$ considering the positive Schur property instead of the Schur property.

Recall that a Banach lattice $E$ has the positive Schur property if positive (or, equivalently, disjoint or positive disjoint) weakly null sequences in $E$ are norm null. This property was introduced by W. Wnuk \cite{wnukglasgow, wnuksurv} and 
F. R\"abiger \cite{rabiger}, and has been extensively studied for many experts, recent developments can be found, e.g., in \cite{ardakani, baklouti, botelhobu, botelholuiz, botlumir,  pedro}. In the next, we give a sufficient condition so that the pair $(E, F)$ has the $WMP^+$ whenever $F$ has the positive Schur property.

\begin{theorem} \label{wmp5}
If $F$ has the positive Schur property, then the pair $(E, F)$ satisfies the $WMP^+$ for every  Banach lattice $E$ such that $B_E$ is sequentially absolutely weakly compact.
\end{theorem}

\begin{proof}
     Suppose that $F$ has the positive Schur property and let $E$ be Banach lattice such that $B_E$ is sequentially absolutely weakly compact. To prove that $(E, F)$ has the $WMP^+$, let $T: E \to F$ be a positive operator and let $(x_n)_n \subset S_{E^+}$ be a non-weakly null sequence such that $\displaystyle \lim_{n \to \infty} \norma{T(x_n)} = \norma{T}$. Since $B_{E}$ is sequentially absolutely weakly compact, there exist a subsequence $(x_{n_k})_k$ and $x \in B_{E}$ such that $x_{n_k} \cvfa x$ in $E$, that is $|x_{n_k} - x| \cvf 0$. Thus $(T|x_{n_k} - x|)_k$ is a positive weakly null sequence in $F$, and by the assumption 
    $$\left |\norma{T(x_{n_k})} - \norma{T(x)} \right | \leq \norma{T(x_{n_k}) - T(x)} \leq \norma{T|x_{n_k} - x|} \to 0,$$
    which implies that $\displaystyle \norma{T(x)} = \lim_{k \to \infty} \norma{T(x_{n_k})} = \norma{T}$.
    Therefore, $T$ is a norm-attaining positive operator, and hence $(E, F)$ has the $WMP^+$. 
\end{proof}

Theorem \ref{wmp5} allows us to present another example of a pair of Banach lattices $(E, F)$ with the $WMP^+$ that fails the $WMP$. Indeed, it was noticed in the proof of \cite[Theorem 3.2]{dantasjungcerv} that $(\ell_2, L_1[0,1])$ cannot have the $WMP$. However, since $L_1[0,1]$ has the positive Schur property and $B_{\ell_2}$ is sequentially absolutely weakly compact set by \cite[Proposition 2.10 and Theorem 2.9]{botlumir}.

As an application of the Eberlein-Smulian theorem, it follows that if $B_E$ is a sequentially absolutely weakly compact set, then it is weakly compact, and hence $E$ is a reflexive Banach lattice. The converse is not necessarily true (see \cite[Example 2.7]{botlumir}).
We conclude this Section by proving a partial converse of Theorem \ref{wmp5}. First, we need to prove a lemma which can be seen as a positive version of the famous Davis – Figiel – Johnson – Pelczynski's theorem (see \cite[Theorem 5.38]{alip}) and as a ``stronger" version of \cite[Theorem 5.42]{alip} for positive operators.

\begin{lemma} \label{wmp6}
    Let $E$ and $F$ be two Banach lattices with $F$ having order continuous norm. If $T: E \to F$ is a positive weakly compact operator, then it factors positively through a reflexive Banach lattice.
\end{lemma}

\begin{proof}
    Letting $W = \sol{T(B_{E^+})}$, we get that $W$ is a convex, solid, and norm bounded subset of $F$ [\cite{oikhbergtursi}, Corollary 19.4(2)]. Thus, by \cite[Theorem 5.41]{alip}, there exists a Banach lattice $\Psi$ which is an ideal of $E$. Moreover, since $T$ is a positive weakly compact operator, $T(B_{E^+})$ is relatively weakly compact subset of $F^+$. Thus the order continuity of the norm in $F$ implies that $W$ is a relatively weakly compact set in $F$ (see \cite[Theorems 4.9 and 4.39(1)]{alip}), which yields by \cite[Theorem 5.37(4)]{alip} that $\Psi$ is reflexive. On the other hand, since $W \subset \Psi$ (see \cite[Theorem 5.37(2)]{alip}), we obtain that the range of $T$ is contained in $\Psi$,
    which allows us to define $S: E \to \Psi$ by $S(x) = T(x)$ for every $x \in E$. In particular, $S$ is a positive operator such that $T = J \circ S$, where $J: \Psi \to F$ is the canonical embedding, proving that $T$ factors positively through the reflexive Banach lattice $\Psi$.
\end{proof}

We conclude our discussion by applying Theorem \ref{teoprinc} and Lemma \ref{wmp6} in order to prove partial converse of Theorem \ref{wmp5}.

\begin{theorem} \label{wmp7}
    Let $F$ be a Banach lattice with order continuous norm.
    If $F$ fails to have the positive Schur property, then there exists a reflexive Banach lattice $E$ such that $(E, F)$ does not have the $WMP^+$.
\end{theorem}

\begin{proof}
If $F$ fails to have the positive Schur property, there exists a positive disjoint weakly null sequence $(y_n)_n \subset S_{F^+}$. Defining $T: \ell_1 \to F$ by $T((a_j)_j) = \sum_{n=1}^\infty a_n y_n$, we obtain that $T$ is a positive weakly compact operator (see \cite[Theorem 5.26]{alip}). Note, also that $(T(e_n) = y_n)_n$ has no convergent subsequence, because it is normalized and weakly null.
Since $F$ has order continuous, we obtain from Lemma \ref{wmp6} that there exist a reflexive Banach lattice $\Psi$ and two positive operators $S: \ell_1 \to \Psi$ and $J: \Psi \to F$ such that $T = J \circ S$. 
Letting, for every $n \in \N$, $S(e_n) = T(e_n) = y_n \in \Psi$, $(J(y_n))_n$ has no convergent subsequence in $F$.
 As $(y_n)_n$ is positive normalized disjoint sequence, it is an unconditional basic sequence. Thus, defining $E = \overline{[y_{n} : n \in \N]}$, we have that $(E, \norma{\cdot}_{\Psi})$ is a reflexive Banach space with an unconditional basis $(y_{n})_n$.
 Letting
$$ \sum_{n=1}^\infty a_n y_{n} \leq_{E} \sum_{n=1}^\infty b_n y_{n} \Leftrightarrow a_n \leq b_n \, \text{ in $\R$ for every $n \in \N$}, $$
we get that $\leq_{E}$ defines a vector lattice structure in $E$. We denote the modulus of $\displaystyle\sum_{n=1}^\infty a_n y_{n}$ with this new order by $\displaystyle\left | \sum_{n=1}^\infty a_n y_{n} \right |_E$.
Let us see that $\norma{\cdot}_{\Psi}$ is a Riesz norm on $E$ with this partial order. Indeed, if 
$$ \sum_{n=1}^\infty |a_n| y_{n}  = \left |\sum_{n=1}^\infty a_n y_{n} \right |_E \leq_{E} \left | \sum_{n=1}^\infty b_n y _{n} \right |_E = \sum_{n=1}^\infty |b_n| y_{n},$$ we have by the definition that $|a_n| \leq |b_n|$ in $\R$ for every $n \in \N$, which implies that
$$ 0 \leq \sum_{n=1}^\infty |a_n| y_{n} \leq \sum_{n=1}^\infty |b_n| y_{n} \quad \text{in $\Psi$,}$$
and hence $\displaystyle \norma{\sum_{n=1}^\infty |a_n| y_{n}}_{\Psi} \leq \norma{\sum_{n=1}^\infty |b_n| y_{n} }_{\Psi}$ [because, $(\Psi, \norma{\cdot}_{\Psi})$ is a Banach lattice]. Consequently, $\displaystyle \norma{| \,\sum_{n=1}^\infty a_n y_{n} \, | }_{\Psi} \leq \norma{| \, \sum_{n=1}^\infty b_n y_{n}\, | }_{\Psi}$, proving that $(E, \norma{\cdot}_{\Psi})$ is a reflexive Banach lattice with the order given by a basis.

Also, $J|_{E}: E \to F$ is a non-compact positive operator, because $(J(y_n))_n$ does not have any convergent subsequence.  It follows by Theorem \ref{teoprinc} that there exists a positive non-norm attaining operator $R: E \to F$. Thus, defining $\widetilde{R}: E \oplus_\infty \R \to F$ by $\widetilde{R}(x, \lambda) = R(x)$ for every $(x, \lambda) \in E \oplus_\infty \R$, we have that $\widetilde{R}$ is a positive operator with $\norma{R} = \norma{\widetilde{R}}$ which does not attain its norm [If $\widetilde{R}$ attains its norm on $(x_0, \lambda_0)$, then $R$ would attain its norm on $x_0$]. However, if $(x_n)_n \subset  S_{E^+}$ is a maximizing sequence for $R$, then $((x_n, 1))_n$ is a positive maximizing sequence of $\widetilde{R}$ which is not weakly null in $E \oplus_\infty \R$. Therefore $E \oplus_\infty \R$ is a reflexive Banach lattice such that $(E \oplus_\infty \R, F)$ fails the $WMP^+$.
\end{proof}

\noindent \textbf{Acknowledgments:} José Lucas P. Luiz is supported by Fapemig Grant APQ-01853-23. Vinícius C. C. Miranda is supported by FAPESP grant and 2023/12916-1 Fapemig Grant APQ-01853-23

\noindent J. L. P. Luiz\\
Instituto Federal do Norte de Minas Gerais\\
Campus de Ara\c cua\'i\\
39.600-00 -- Ara\c cua\'i -- Brazil\\
e-mail: lucasvt09@hotmail.com

\medskip

\noindent V. C. C. Miranda\\
Centro de Matem\'atica, Computa\c c\~ao e Cogni\c c\~ao \\
Universidade Federal do ABC \\
09.210-580 -- Santo Andr\'e -- 		Brazil.  \\
e-mail: colferaiv@gmail.com

\end{document}